[1994/12/01]
\documentclass[12pt, reqno]{amsart}
\usepackage{a4wide}
\usepackage[english, activeacute]{babel}
\usepackage{amsmath,amsthm,amsxtra}
\usepackage{epsfig}
\usepackage{amssymb}
\usepackage{latexsym}
\usepackage{amsfonts}
\newtheorem*{TA}{Theorem A}
\usepackage[colorlinks=true]{hyperref}
\hypersetup{linkcolor=red,citecolor=blue,filecolor=dullmagenta,urlcolor=red} 

\title{On the long time behavior of solutions to the Intermediate Long Wave equation}
\author{Claudio Mu\~noz}

\author{Gustavo Ponce}

\author{Jean-Claude Saut}
\thanks{CM was partially supported by Fondecyt no. 1191412, Ecos-Sud C18E06 and Conicyt PIA AFB170001. J.-C. S. is partially supported by the ANR Project ANuI}
\address{CNRS and Departamento de Ingenier\'ia Matem\'atica DIM-CMM UMI 2807-CNRS \\ Universidad de Chile, Santiago, Chile}
\email{cmunoz@dim.uchile.cl}
\address{Laboratoire de Mathématiques, UMR 8628 Université Paris-Sud et CNRS, 91405 Orsay, France}
\email{jean-claude.saut@u-psud.fr}
\address{Department  of Mathematics, 
University of California, Santa Barbara, CA 93106. USA}
\email{ponce@math.ucsb.edu}
\date{\today}

\chardef\bslash=`\\ 

\newtheorem{thm}{Theorem}[section]
\newtheorem{cor}[thm]{Corollary}
\newtheorem{lem}[thm]{Lemma}

\theoremstyle{remark}
\newtheorem{rem}{Remark}[section]

\numberwithin{equation}{section}

\newcommand{\R}{\mathbb{R}}

\newcommand{\Z}{\mathbb{Z}}

\newcommand{\la}{\lambda}

\newcommand{\sgn}{\operatorname{sgn}}

\newcommand{\be}{\begin{equation}}
\newcommand{\ee}{\end{equation}}
\newcommand{\bp}{\begin{proof}}
\newcommand{\ep}{\end{proof}}
\newcommand{\bel}{\begin{equation}\label}
\newcommand{\eeq}{\end{equation}}
\newcommand{\bea}{\begin{eqnarray}}
\newcommand{\eea}{\end{eqnarray}}
\newcommand{\bee}{\begin{eqnarray*}}
\newcommand{\eee}{\end{eqnarray*}}
\newcommand{\ben}{\begin{enumerate}}
\newcommand{\een}{\end{enumerate}}

 \providecommand{\abs}[1]{\lvert#1 \rvert}

\newcommand{\ve}{\varepsilon}

\newcommand{\eval}[2][\right]{\relax
  \ifx#1\right\relax \left.\fi#2#1\rvert}

\let\abs=\envert

\begin{document}
\begin{abstract}
We show that the limit infimum, as time $\,t\,$  goes to infinity, of any uniformly bounded in time $H^{3/2+}\cap L^1$ solution to the Intermediate Long Wave equation converge to zero locally in an increasing-in-time region of space of order $\,t/\log(t)$. Also, for solutions with a mild $L^1$-norm growth in time is established that its limit infimum converge to zero, as time goes to infinity. This confirms the non existence of breathers and other solutions for the ILW model moving with a speed \lq\lq slower" than a soliton. We also prove that in the far field linearly dominated region, the $L^2$ norm of the solution also converges to zero as time approaches infinity.

In addition, we deduced several scenarios for which  the initial value problem associated to the generalized Benjamin-Ono and the generalized Intermediate Long Wave equations cannot possess time periodic solutions (breathers).

Finally, as it was previously demonstrated in solutions of the KdV and BO equations,  we establish the following propagation of regularity result : if the datum $u_0\in H^{3/2+}(\R)\cap H^m((x_0,\infty))$, for some $\;x_0\in\R,\,m\in Z^+,\,m\geq 2$, then the corresponding solution $u(\cdot,t)$ of the Intermediate Long Wave equation  belongs to $H^m(\beta,\infty)$, for any $t>0$ and  $\beta\in\R$.

\end{abstract}
\maketitle \markboth{ILW smoothing and decay estimates}{C. Mu\~noz, G. Ponce and J.-C. Saut}
\renewcommand{\sectionmark}[1]{}

\section{Introduction and main results}


We consider the Intermediate Long Wave (ILW) equation in 1D for $\delta>0$,
\be\label{ILW} 
\begin{aligned}
 \partial_t u + \mathcal T_\delta \partial_x^2  u + \frac1{\delta}\partial_x u  + u\partial_x u = 0,\qquad (t,x) \in & ~ \R\times \R.
\end{aligned}
\ee
Here and along this paper $u=u(t,x) \in \R$ is a real-valued function, and
\be\label{T}
\mathcal T_\delta f(x) := -\, \frac1{2\delta} \,\hbox{p.v.} \int \coth\left(\frac{\pi(x-y)}{2\delta}\right)f(y)dy.
\ee
Note that $\mathcal T_\delta$ is a order zero Fourier multiplier, in the sense that $\partial_x\mathcal T_\delta$ is the multiplier with symbol
\be\label{1.3} 
\sigma(\partial_x\mathcal T_\delta)=\widehat{\partial_x\mathcal T_\delta} = - 2\pi \xi \,\hbox{coth}\,(2\pi \delta \xi).
\ee
The equation ILW \eqref{ILW} describes the long internal gravity waves in a two layers stratified fluid, the lower layer having a large finite depth represented by the parameter $\,\delta$, see 
 \cite{Jo}, \cite{KKD}, \cite{JE}, \cite{CL}, \cite{SAK} for a formal derivation and \cite{BLS, CGK} for a rigorous derivation in the sense of consistency.
 
 \medskip
 
In \cite{ABFS} it was proven that solutions of the ILW  as $\delta \to \infty$ (infinite depth limit) converge to solutions of the Benjamin-Ono (BO) equation, see \cite{Be}, \cite{On}
 \be\label{BO} 
\partial_t u + \mathcal H\partial_x^2  u   + u\partial_x u = 0,
\ee
with the same initial data, here $\mathcal H$ stands for the Hilbert transform
\begin{equation}
\label{hita}
\begin{aligned}
\mathcal H f(x) :=&~ \frac{1}{\pi} {\rm p.v.}\Big(\frac{1}{x}\ast f\Big)(x)
\\
:=& ~ \frac{1}{\pi}\lim_{\epsilon\downarrow 0}\int\limits_{|y|\ge \epsilon} \frac{f(x-y)}{y}\,dy=(-i\,\sgn(\xi) \widehat{f}(\xi))^{\vee}(x).
\end{aligned}
\end{equation}

Also, in \cite{ABFS} it was shown that  if $u_{\delta}(t,x)$ denotes the solution of the ILW equation \eqref{ILW}, then
\be
\label{scaleKdV}
v_{\delta}(t,x)=\,\frac{3}{\delta} \,u_{\delta}\left(x,\frac{3}{\delta} t\right)
\ee
converges  as $\delta\to 0$ (shallow-water limit) to the solution  of the Korteweg-de Vries (KdV) equation \cite{KdV}
\be\label{KdV} 
\partial_t u + \partial_x^3  u  + u\partial_x u = 0,
\ee
with the same initial data.

\medskip

The ILW equation has been proven to be complete integrable.  In fact the Inverse Scattering formalism was given in \cite{KAS, KSA}, where one finds in particular the Lax pair of the ILW equation, but  no rigorous theory for solving the Cauchy problem by this method is known (see however  \cite{KPW} for recent progress on the direct scattering problem). For further comments on general properties of the ILW equation we refer to \cite{JE}, \cite{KAS}, \cite{KKD}, \cite{SAF}, and a recent survey \cite{Sa}.

\medskip
Note also that the ILW equation is obtained as a one-dimensional,  uni-directional reduction of a class of ILW {\it systems} derived in \cite{BLS, CGK}. Long time existence results for those systems were obtained in \cite{X} but the global existence of even small solutions is still unknown to our knowledge.

\medskip
Formally,  real solutions of \eqref{ILW} satisfy  infinitely many conservation laws, due to its integrability, (see {\it eg} \cite{LR}, \cite{Mat3}) .  The first three are the following  ones, and are only due to the hamiltonian nature of the ILW equation.
\begin{equation}
\label{CL}
\begin{aligned}
I_1(u)&:=\int u(t,x)dx=I_1(u_0),\quad \;\;\;I_2(u):=\int u^2(t,x)dx=I_2(u_0),\\
I_3(u)&:=\int\left( u\mathcal T_\delta \partial_x u +\frac1{\delta} u^2 -\frac13u^3\right)(t,x)dx=I_3(u_0)\\
&~ \qquad  \text{(conservation of the Hamiltonian)}.
\end{aligned}
\end{equation}

The first non-trivial invariant controls the $H^1$ norm of the solution and is given by

\begin{equation}\begin{split} 
I_4(u)=   &\int_{-\infty}^\infty \Bigg( \frac{1}{4}u^4+\frac{3}{2}u^2\mathcal T_\delta(u_x)+\frac{1}{2}u_x^2+\frac{3}{2}[\mathcal T_\delta(u_x)]^2\\
&\qquad \quad +\frac{1}{\delta} \left( \frac{3}{2}u^3+\frac{9}{2}u\mathcal T_\delta(u_x) \right) +\frac{3}{2\delta^2} \Bigg)dx.
\end{split}
\end{equation}

The next invariant controls the $H^{3/2}$ norm of the solution. The general structure of the invariants is given for instance in \cite{ABFS}.

The  global well-posedness of the initial value problem (IVP) associated to the ILW equation \eqref{ILW} was studied in \cite{ABFS} where the following result was established:

\begin{TA}\label{WP} \cite{ABFS}
For any initial data $u_0\in H^s(\R),\,s>3/2$, for any $\,T>0$ and any $\delta>0$ the corresponding IVP associated to the ILP equation \eqref{ILW} has a unique solution
\begin{equation}\label{clase}
u\in C([0,T]:H^s(\R))\cap C^1([0,T]:H^{s-2}(\R))\cap L^{\infty}([0,T]:H^{s+1/2}_{loc}(\R)).
\end{equation}

Moreover, the map data $\to$ solution is locally continuous from $H^s(\R)$ to the class defined in \eqref{clase} and the $H^n(\R), n=0,\frac{1}{2}, 1, \frac{3}{2}$ norms of $u(\cdot,t)$ are uniformly bounded in time.
\end{TA}

\begin{rem}
The proof uses the local $H^s(\R), s>3/2$ theory, the uniform $H^{3/2}(\R)$ bound (thus integrability) and an interpolation argument.
\end{rem}

\vskip.1in

Stronger results has been obtained in \cite{MoVe}, where (unconditional) local well-posedness was established in $\,H^s(\R)$ with $s\geq 1/2$. Due to the conservation of the hamiltonian, this result  implies the global well-posedness of ILW in the energy space $H^{1/2}(\R).$


\vspace{0.3cm}
It is well-known that  \eqref{ILW} possesses  soliton (or solitary wave) solutions of the form \cite{Jo}
\be\label{Soliton}
u(t,x)= Q_{\delta,c}(x-ct),  \quad c>\frac1{\delta},
\ee
where $Q_{\delta,c}$  solves
\begin{equation}
\label{solitons}
 \partial_x\mathcal T_\delta  Q_{\delta,c}  +\left( \frac1{\delta} -c\right) Q_{\delta,c}  + \frac12 Q_{\delta,c}^2=0.
\end{equation}
We remark  that $Q_{\delta,c}$ is exponentially localized in space contrary to the BO soliton, and it is given by the formula \cite{AT}
\be\label{Solsol}
Q_{\delta,c} (s):= \frac{b(c)\sin(a(c)\delta)}{\cos(a(c)\delta) + \cosh(a(c)s)}, \quad a,b \hbox{ depending on $c$.}
\ee
The uniqueness (up to translation) of the ILW soliton is proven in \cite{A,AT}. Because of its integrability, the ILW equation possesses also multi-solitons, see \cite{JE}.
\medskip

Let $C>0$ be an arbitrary constant and  $I_{b}(t)$ be the time-depending interval 
\be\label{Interval}
I_{b}(t):= \left( - \frac{C |t|^{b}}{\log |t|} , \frac{C |t|^{b}}{\log |t|} \right),  \quad |t|\geq 10,\quad b\in(0,1].
\ee

Our first result here is the following : 
 
\begin{thm}\label{TH1}
Let $u=u(t,x)$ be a solution to \eqref{ILW} such that
 \begin{equation}
 \label{m1}
 u\in C(\R:H^s(\R))\cap C^1([0,T]:H^{s-2}(\R))\cap L^{\infty}([0,T]:H^{s+1/2}_{loc}(\R)),\;\;\;s>3/2.
 \ee
 Assume that
 \be
 \label{hyp2}
 u\in L^\infty_{loc}(\R:L^1(\R))
\end{equation}
such that for some $a\in [0,\frac12)$ and $c=c(\delta)>0$
\be\label{Growth}
\sup_{t\in[0,T]} \int |u(t,x)|dx \leq c\langle T \rangle^a=c\,(1+T^2)^{a/2}.
\ee
Then for $\;b=1-a,$
\be\label{AS}
\liminf_{t\to \infty} \int_{|x|\leq \frac{C\,t^{b}}{\log(t)}}\;
(u^2+( q_{\delta}(\partial_x)u)^2)(t,x)dx=0.
\ee
where $q_{\delta}(\partial_x)$ is a Fourier multiplier with even  symbol $q_{\delta}(\xi)$, $q_{\delta}(\xi)>0$ for $\xi\neq 0$ with $q_{\delta}(0)=0$, $q_{\delta}'(0^+)=\delta^{-1/2}$, and  
\be\label{1234}
q_{\delta}(\xi)=\frac{\xi}{\delta^{1/2}}+O\Big(\frac{\xi^3}{\delta^{5/2}}\Big)\;\;\text{as}\;\;\xi\downarrow 0,\;\;\;q_{\delta}(\xi)=\sqrt{\xi}+O\Big(\frac{1}{\delta{\xi}^{1/2}}\Big)\;
\;\text{as}\;\;\xi\uparrow \infty.
\ee
Therefore, no soliton nor breather solution exists for ILW inside the region $I(t)$, for any time $t$ sufficiently large.
\end{thm}

\begin{rem}
The result in \eqref{AS} also holds as $t\to -\infty$.
\end{rem}

\begin{rem}
Related with the present problem we have the variants found in \cite{KMM}, \cite{KMM1}, \cite{KMPP}. However, here as in \cite{MuPo2}  even under the strongest hypothesis $\,a=0$ in \eqref{Growth} one needs a weight outside the cut-off function $\phi(\cdot)$. This can be seen as a consequence of the weak dispersive relation in the ILW equation which  does not allow to apply the  argument in \cite{KMM1} and  in \cite{MuPo1}. This weight  lets us to close the estimate in a weaker form than those obtained in \cite{KMM1} and \cite{MuPo1}, involving the lim inf instead of the lim.

\end{rem}

\begin{rem}
Theorem \ref{TH1}  affirms that there exists a sequence of times $\{t_n\,:\,n\in\Z^+\}$ with $t_n\uparrow \infty$ as $n\uparrow \infty$ such that
\be\label{asymp22}
\lim_{n\uparrow \infty} \,\int_{|x|\leq \frac{C\,t^{b}_n}{\log(t_n)}}\;
(u^2+( q(\partial_x)u)^2)(x,t_n)dx=0.
\ee

\end{rem}

\begin{rem}

It is known that  the generalized KdV equation
\be
\label{gjkdv}
 \partial_t u +  \partial_x^3  u + \partial_x (g(u)) = 0,\qquad (t,x) \in  \, \R\times \R,
 \ee
possesses breather solutions, i.e. localized solutions which are  also periodic in time, for $\;g(u)=u^3\,$ (modified KdV) and for $\,g(u)=u^2+\mu u^3$  (Gardner equation) with $\mu>0$, see \cite{MuPo1}. In the case of the modified KdV these breather solutions can have arbitrarily small energy i.e. $H^1$-norm.

\medskip

As in the case of the generalized BO equation, for the generalized ILW equations the energy is related to the $\,H^{1/2}$-norm.
 
\end{rem}

\vskip.1in
Next, we consider the initial value problem for the generalized Benjamin-Ono (gBO) equation
\be\label{BOg} 
\begin{aligned}
\begin{cases}
& \partial_t u + \mathcal H \partial_x^2  u + \partial_x (u^k +p_k(u)) = 0,\qquad (t,x) \in  \, \R\times \R,\;\;k=2,3,4,...\\
&u(x,0)=0, 
 \end{cases}
\end{aligned}
\ee
where
\be
\label{poly}
p_k(s)=\sum_{j=k+1}^m a_j\,s^j\,,\;\;\;\;\;\;\;\;\;m\in\Z^+,\,\;\;\;a_j\in\R,\;\;\;\;k=2,3,....
\ee

\vskip.1in
The following results are concerned with the non-existence of time periodic solutions to the IVP associated to the gBO equation \eqref{BOg}:
\vskip.1in

\begin{thm}\label{TH1a}

Let 
\be
\label{h1}
f(u):= u^k+p_k(u),
\ee
with $\,k\,$ and $\,p_k(\cdot)$ as above.
If a real non-trivial solution $u=u(t,x)$ of the IVP \eqref{BOg} 
 \be
 \label{class1}
 u\in C(\R:H^4(\R)\cap L^2(|x|^4dx)) \cap C^1(\R:H^2(\R)\cap L^2(|x|^2dx))
\end{equation}
is periodic in time with period $\,\omega$, then
\be
\label{h2}
\int_0^{\omega}\int_{-\infty}^{\infty}\,f(u(t,x))dxdt=0.
\ee
In particular, if
\be
\label{h3}
\begin{cases}
\begin{aligned}
&f(u)=u^{2j},\;\;\;\;\;\;j\in\Z^+,\\
&\text {or}\\
&f(u)=u^{2j}\pm \rho u^{2j+1}+\epsilon^2u^{2j+2},\;\;\;\;\;\;\;0\leq \rho<2\epsilon,\;\;\;\;\;\;\;j\in\Z^+,
\end{aligned}
\end{cases}
\ee
then $u=u(t,x)$ cannot be time periodic.

\end{thm}

\begin{thm}\label{TH2a}

 For any integer $\,k\geq 3\,$ and any polynomial $\,p_k(\cdot)\,$ in \eqref{poly} there exists $\,\epsilon=\epsilon(k;p_k)>0$ such that
any real solution $u=u(t,x)$ of the IVP \eqref{BOg} 
 \be
 \label{class2}
 u\in C(\R:H^4(\R)\cap L^2(|x|^4dx)) \cap C^1(\R:H^2(\R)\cap L^2(|x|^2dx))
\end{equation}
corresponding to an initial data $u_0$ with 
\be
\label{hy}
0\neq \| u_0\|_{1/2,2} < \epsilon
\ee
cannot be  time periodic.
\end{thm}

\begin{rem}

(i)  Real solution of the IVP associated to the equation \eqref{BOg} satisfy at least three conservation laws
\be
\label{cl222}
\begin{aligned}
&\mathcal I_1(u):=\int u(t,x)dx=\mathcal I_1(u_0),\;\;\;\;\;\mathcal I_2(u):=\int u^2(t,x)dx=\mathcal I_2(u_0),\\
&\mathcal I_3(u):=\int_{-\infty}^{\infty}\Big(\frac{(D_x^{1/2}u)^2}{2}-\Big(\frac{u^{k+1}}{k+1}+G_k(u)\Big)\Big)dx=\mathcal I_3(u_0),
\end{aligned}
\ee
where
\be
\label{f122}
G_k(s) =\int_0^s\,p_k(r)dr.
\ee

\vskip.1in

(ii) From the conservation laws we have control of the $L^p$-norm of the small solution for $\,p\in[2,\infty)$ but not for $p=\infty$. This is the reason to consider only polynomial perturbations in \eqref{BOg} and \eqref{poly}.

\vskip.1in
(iii) The assumption \eqref{class1} and \eqref{class2}, which justify the integrations by part in the proof,  is not optimal.
\end{rem}
\vskip.1in

Consider now that IVP associated to the generalized ILW (gILW) quation 
\be\label{gILW} 
\begin{cases}
\begin{aligned}
& \partial_t u + \mathcal T_\delta \partial_x^2  u + \frac1{\delta}\partial_x u  + \partial_x( u^k+p_k(u)) = 0,\qquad (t,x) \in & ~ \R\times \R,\;\;k=2,3,4,..\\
& u(x,0)=u_0(x),
\end{aligned}
\end{cases}
\ee
with $\,p_k(\cdot)$ as in \eqref{poly}. Similarly, to the results in Theorem \ref{TH1a} for the gBO equation for the gILW \eqref{gILW} we have:

\begin{thm}\label{TH7}

Let 
\be
\label{h11}
f(u):= u^k+p_k(u),
\ee
with $\,k\,$ and $\,p_k(\cdot)$ as in \eqref{poly}.
If a real non-trivial solution $u=u(t,x)$ of the IVP \eqref{gILW} 
 \be
 \label{class11}
 u\in C(\R:H^4(\R)\cap L^2(|x|^4dx)) \cap C^1(\R:H^2(\R)\cap L^2(|x|^2dx))
\end{equation}
is periodic in time with period $\,\omega$, then
\be
\label{h12}
\int_0^{\omega}\int_{-\infty}^{\infty}\,f(u(t,x))dxdt=0.
\ee
In particular, if
\be
\label{h13}
\begin{cases}
\begin{aligned}
&f(u)=u^{2j},\;\;\;\;\;\;j\in\Z^+,\\
&\text {or}\\
&f(u)=u^{2j}\pm \rho u^{2j+1}+\epsilon^2u^{2j+2},\;\;\;\;\;\;\;0\leq \rho<2\epsilon,\;\;\;\;\;\;\;j\in\Z^+,
\end{aligned}
\end{cases}
\ee
then $u=u(t,x)$ cannot be time periodic.

\end{thm}

From the arguments found in the proof of Theorem \ref{TH1}  one has that  the proof of Theorem \ref{TH7} becomes similar to that for Theorem \ref{TH1a}. hence it will be omitted.

\medskip

Now, we shall see that the propagation of regularity established in \cite{ILP1} and \cite{ILP2} in solutions of the KdV equations and BO equations respectively also holds in solutions of the ILW equation \eqref{ILW}.

\begin{thm}\label{TH2} Let $u_0\in H^s(\R),\,s>3/2$  and 
$$
 u\in C(\R:H^s(\R))\cap C^1([0,T]:H^{s-2}(\R))\cap L^{\infty}([0,T]:H^{s+1/2}_{loc}(\R)) ,
 $$
be  the corresponding solution of  the IVP associated to the ILW equation \eqref{ILW} with data $u_0$ provided by Theorem \eqref{WP}.
If for some $x_0\in \R, \,m\in \Z^+$ and $\,m>s$ one has that
\be
\label{hyp1}
M_0:=\,\int_{x_0}^{\infty} (\partial_x^mu_0(x))^2dx<\infty,
\ee
then for any $\gamma>0, \,T>0,\,\epsilon>0,\,R>5\epsilon$ it follows that
\be
\label{res1}
\begin{aligned}
&\sup_{0\leq t\leq T}\,\int_{x_0+\epsilon-\gamma t} (\partial_x^mu(t,x))^2dx+ \int_0^T \int_{x_0+\epsilon-\gamma t}^{x_0+R-\gamma t}(D_x^{1/2}\partial_x^mu(t,x))^2\,dx \,dt\\
\\
&\qquad \leq c=c(\|u_0\|_{s,2};T;M_0;\gamma;\epsilon;R).
\end{aligned}
\ee
\end{thm}

\vskip.1in

Our scheme  to prove Theorem \ref{TH2} will be to reduce it to the proof  provided in \cite{ILP2} for solutions of the BO equation.

\vskip.1in

Our last theorem concerns with the \emph{far field region dominated by linear decay}. A direct plane wave analysis of \eqref{ILW} reveals that, formally, linear waves tends to travel with nonnegative speeds, similar to KdV and BO, and they concentrate in the spatial region $\{x\leq 0\}$. Here we prove that, despite the size of the data, the region of truly influence of linear waves is just $\{ -t\log^{1+\epsilon} t \ll  x\leq 0 \}$, for $\epsilon>0$.

\begin{thm}\label{TH3} Assume $u\in  C(\R:H^{1/2+}(\R))\cap L^{\infty}(\R:H^{1/2+}(\R))$. Then for any $\mu(t) \gtrsim t\log^{1+\epsilon} t $ and $\epsilon>0$,
\be\label{decay_far}
\lim_{t\to+\infty} \| u(t)\|_{L^2(|x|\sim \mu(t))} =0.
\ee
\end{thm}

\begin{rem}
From the proof of Theorem \ref{TH3}, it will become clear that it is also valid for KdV, mKdV, BBM and BO equations. Therefore, this estimate is independent of the integrability of the equation. Note also that no size restriction is needed in \eqref{decay_far}, therefore it is a truly nonlinear decay estimate. It can be also rephrased as \emph{complete integrability cannot travel/move faster than $t\log t$}. 
\end{rem}

\begin{rem}
Theorem \ref{TH3} seems to us the first decay result in the energy space inside the linearly dispersive regime using nonintegrable techniques, with data only in the energy space.
\end{rem}

An interesting by-product of Theorem \ref{TH3} is the following result, in principle awkward in nature, that essentially says that \emph{integrability in time does not necessarily imply decay to zero in time}. 

\begin{cor}\label{TH4}
Under the hypothesis of Theorem \ref{TH3}, for any $\epsilon>0$ and any $\tilde\phi \in C_0^\infty([0,1])$, $\tilde\phi>0$ in $(0,1)$, one has
\be\label{LL}
\int_{10}^\infty \frac{1}{t}  \left( \int_{-t}^{t\log^{1+\epsilon} t-t} \left(\frac{x+t}{t \log^{1+\epsilon } t}\right)\tilde \phi\Big( \frac{x+t}{t \log^{1+\epsilon } t}\ \Big) u^2(t,x)dx \right) dt<+\infty.
\ee
\end{cor}

We emphasize that Corollary \ref{TH4} is against intuition, because it shows integrability in time of a portion of the $L^2$ norm on a region where we know that solitons appear. For instance, \eqref{LL} must hold for any fixed soliton $Q_{\delta,c}$ described in \eqref{Soliton}, and it is indeed the case. See Remark \ref{Final} for more details.

\medskip

The rest of this paper is organized as follows. In section 2 we shall prove Theorem \ref{TH1}. Section 3 contains the proof of Theorem
 \ref{TH1a} and Theorem \ref{TH2a}, Section 4 that  of Theorem  \ref{TH2}, and Section 5 the proof of Theorem \ref{TH3} and Corollary \ref{TH4}.

\subsection*{Acknowledgments} We would like to thank C. Kwak for interesting comments and discussions about this work.

\bigskip

\section{Proof of Theorem \ref{TH1}}

\medskip

First, we shall obtain some estimates concerning the symbol modeling the dispersive relation in the ILW equation
\eqref{ILW}
\be
\label{a1}
\sigma\Big( \mathcal T_{\delta} \partial_x^2  +\frac{\partial_x}{\delta}\Big)(\xi)= - i\left( (2\pi\xi)^2\,\hbox{coth}(2\pi\delta\xi)-\frac{2\pi\xi}{\delta}\right)= - i\,\Omega_{\delta}(2\pi\xi),
\ee
with
\be
\label{symbol2}
 \mathcal T_\delta \partial_x^2  +\frac{\partial_x}{\delta}= - \mathcal L_{\delta}(\partial_x)\,\partial_x.
 \ee

We observe that $\Omega_{\delta}(\cdot)$ is an smooth (analytic), real valued odd function with $\Omega_{\delta}(\xi)>0$ for $\xi>0$,
\begin{equation}
\label{a2}
\Omega_{\delta}(\xi)=\frac{1}{3} \frac{\xi^3}{\delta}+O\Big(\frac{\xi^5}{\delta^3}\Big)\;\;\;\;\;\text{as}\;\;\;\;\;|\xi|\downarrow 0,
\end{equation}
and
\begin{equation}
\label{a3}
\Omega_{\delta}(\xi)=\xi|\xi|-\frac{\xi}{\delta}+O(1)\;\;\;\;\;\text{as}\;\;\;\;\;|\xi|\uparrow \infty.
\end{equation}

Moreover, 
\begin{equation}
\label{a4}
\Omega'_{\delta}(\xi)= \frac{\xi^2}{\delta}+O\Big(\frac{\xi^4}{\delta^3}\Big)\;\;\;\;\;\text{as}\;\;\;\;\;|\xi|\downarrow 0,
\end{equation}
and
\begin{equation}
\label{a5}
\Omega'_{\delta}(\xi)=2|\xi|-\frac{1}{\delta}+O\Big(\frac{1}{\delta\xi}\Big)\;\;\;\;\;\text{as}\;\;\;\;\;|\xi|\uparrow \infty.
\end{equation}

Also, $\,\Omega_{\delta}'(\cdot)$ is even with $\Omega_{\delta}'(0)=\Omega_{\delta}''(0)=0$, $\Omega_{\delta}'(\xi)>0$ for $\xi> 0$, and 
\be
\label{a6}
\| \Omega_{\delta}''\,\|_{\infty}\leq c=c(\delta).
\ee

Consequently, the symbol $\mathcal L_{\delta}(\xi)=\xi\coth(\delta\xi)-1$ of the operator $\mathcal L_{\delta}(\partial_x)$ in \eqref{symbol2}  is smooth, even, positive for $\xi \neq 0$ with
\be
\label{123}
\mathcal L_{\delta}(0)=\mathcal L_{\delta}'(0)=0\;\;\;\;\text{and}\;\;\;\;\mathcal L_{\delta}''(0)=\frac{2}{3\delta},
\ee
and
\be
\label{123b}
\mathcal L_{\delta}(\xi)=\frac{\xi^2}{3\delta}+O(1)\;\;\;\;\;\text{as}\;\;\;\;\;\xi\uparrow \infty
\ee

Next, we define the operator $\,q_{\delta}(\partial_x)$ whose symbol $\,q_{\delta}(\xi)$ is the square root of  $\Omega_{\delta}'(\xi)$, i.e.
\be
\label{a7}
\Omega'_{\delta}(\xi)=q_{\delta}(\xi)\,q_{\delta}(\xi),
\ee
such that $q_{\delta}(\cdot)$ is even, $q_{\delta}(\xi)>0$ for $\xi>0$. Hence,
\begin{equation}
\label{a8}
q_{\delta}(\xi)= \frac{\xi}{\delta^{1/2}}+O\Big(\frac{\xi^3}{\delta^{5/2}}\Big)\;\;\;\;\;\text{as}\;\;\;\;\;\xi\downarrow 0,
\end{equation}
and
\begin{equation}
\label{a9}
q_{\delta}(\xi)=\sqrt{\xi}\Big(1-\frac{1}{2\delta \xi}+O(\delta\xi^{-2})\Big)\;\;\;\;\;\text{as}\;\;\;\;\;\xi\uparrow \infty,
\end{equation}
with
\be
\label{a10}
q_{\delta}(0)=0\;\;\;\;\;\;\text{and}\;\;\;\;\;\;q_{\delta}'(0^+)=\frac{1}{\delta^{1/2}}.
\ee
In addition, we shall use that
\be
\label{a11}
\mathcal T_\delta \partial_x^2   +\frac{\partial_x}{\delta}=\partial_x^2\Psi_{\delta}(\partial_x).
\ee
where the symbol of $\,i\,\Psi_{\delta}(\cdot)$, $i\Psi(\xi)=1/\xi-\coth(\delta\xi)$,  is $C(\R) \cap L^{\infty}(\R)$, odd, real valued and of order zero.

\vskip.1in
\noindent
Without loss of generality, we assume now that $t\geq 10$ in \eqref{Interval}. We recall that $a+b=1,\,a,b\geq0$. Let
\be\label{la}
\la(t) := \frac{t^{b}}{\log t},\qquad \mu(t):=t^a \log^2 t.
\ee
Note that
\[
\frac{\la'(t)}{\la(t)}\sim \frac{\mu'(t)}{\mu(t)}\sim \frac1t, \quad \la(t)\mu(t) =t^{a+b}\log t.
\]
This follows from
\[
\la'(t) = \frac{b}{t^{1-b} \log t} -\frac1{t^{b} \log^2 t} = \frac1{t^{1-b} \log t} \left(b - \frac1{\log t}\right),
\]
and
\be\label{Computations}
\frac{\la'(t)}{\la(t)} = \frac{1}{t}\left( \frac12 - \frac1{\log t} \right), \quad \la'^2(t) =  \frac{1}{t \log^2 t}\left(\frac12 - \frac 1{\log t} \right)^2.
\ee

We introduce the following class of  weight
\be
\label{class}
\begin{aligned}
\mathcal A_C=\{\phi:\R\to\R &\,:\,\,\phi\;\text{smooth}, \,\phi(-C-1) =0,\,\text{supp} \,\phi'\subset [-C-1,C+1],\\
&\,\;\;\;\;\phi'\,\text{even},\,\;\phi'(x)=1\,\;\text{for}\,\;|x|\leq C,\;\phi'(x)\leq 0\,\;\text{for}\,\;x\geq 0\,\}.
\end{aligned}
\ee

We shall use that if $\phi\in \mathcal A_C$, then there exists $\phi_0\in \mathcal A_C$ such that
\be
\label{issue}
\phi'(x)=\big(\phi_0'(x)\big)^{3}.
\ee

This choice of weights is not essential but it will simplify the exposition. For each $\phi\in \mathcal A_C$ we  consider the functional $\mathcal V(t)=
\mathcal V_{\phi}(t)$ for $u$ solving \eqref{ILW}
\be\label{III}
\mathcal V(t):=\frac1{\mu(t)}\int \phi \Big( \frac{x}{\la(t)} \Big)  u(t,x)dx.
\ee
Notice that from the hypothesis \eqref{Growth} of Theorem \ref{TH1}, one has  
\be
\label{hypo11}
\sup_{t\in \R}\mathcal V(t) <+\infty.
\ee

\begin{lem}\label{dtI0}
We have
\be\label{dtI}
\begin{aligned}
\frac d{dt}\mathcal V(t) = &~ {}  - \frac{\la'(t)}{\mu(t)\la(t)} \int \frac{x}{\la(t)}  \phi'\Big( \frac{x}{\la(t)} \Big) u(t,x)dx \\
& ~{} - \frac{\mu'(t)}{\mu^2(t)}\int   \phi \Big( \frac{x}{\la(t)} \Big)  u(t,x)dx \\
& ~{} -\frac1{\mu(t)\lambda^2(t)}\int  \phi'' \Big( \frac{x}{\la(t)} \Big)   \Psi_{\delta}(\partial_x) u(t,x)dx \\
& ~{} + \frac1{2\mu(t)\lambda(t)}\int   \phi' \Big( \frac{x}{\la(t)} \Big)  u^2 (t,x)dx.
\end{aligned}
\ee
\end{lem}
\begin{proof}
We have, using \eqref{a11}, that 
\[
\begin{aligned}
\frac d{dt}\mathcal V(t) = &~ {} - \frac{\la'(t)}{\mu(t)\la(t)} \int \frac{x}{\la(t)}  \phi'\Big( \frac{x}{\la(t)} \Big) u(t,x)dx - \frac{\mu'(t)}{\mu^2(t)}\int   \phi \Big( \frac{x}{\la(t)} \Big)  u(t,x)dx\\
&~ {}  + \frac1{\mu(t)}\int   \phi \Big( \frac{x}{\la(t)} \Big) \partial_t u(t,x)dx\\
= &~ {} - \frac{\la'(t)}{\mu(t)\la(t)} \int \frac{x}{\la(t)}  \phi'\Big( \frac{x}{\la(t)} \Big) u(t,x)dx  - \frac{\mu'(t)}{\mu^2(t)}\int   \phi \Big( \frac{x}{\la(t)} \Big)  u(t,x)dx \\
&~ {}  - \frac1{\mu(t)\lambda(t)}\int   \phi \Big( \frac{x}{\la(t)} \Big)\Psi_{\delta}(\partial_x) \partial_x^2  u (t,x)dx+ \frac1{2\mu(t)\lambda(t)}\int   \phi' \Big( \frac{x}{\la(t)} \Big)  u^2 (t,x)dx\\
= &~ {} - \frac{\la'(t)}{\mu(t)\la(t)} \int \frac{x}{\la(t)}  \phi'\Big( \frac{x}{\la(t)} \Big) u(t,x)dx - \frac{\mu'(t)}{\mu^2(t)}\int   \phi \Big( \frac{x}{\la(t)} \Big)  u(t,x)dx \\
& ~{}- \frac1{\mu(t)\lambda^2(t)}\int  \phi'' \Big( \frac{x}{\la(t)} \Big)    \Psi_\delta(\partial_x)u(t,x)dx \\
& ~{} + \frac1{2\mu(t)\lambda(t)}\int   \phi' \Big( \frac{x}{\la(t)} \Big)  u^2 (t,x)dx.  
 \end{aligned}
\]
\end{proof}

Combining  the fact that  $\Psi_{\delta}$ is a $C^1$ symbol of order zero with the second conservation law in \eqref{CL} one gets that
\be
\label{a12}
\Big| \frac{1}{\mu(t)\lambda^2(t)}\int  \phi'' \Big( \frac{x}{\lambda(t)} \Big)    \Psi_{\delta}(\partial_x)u(t,x)dx \Big |
\leq\frac{c}{\mu(t)\lambda^{3/2}}
=\frac{c}{t^{a+3b/2}(\log(t))^{1/2}}.
\ee
From the hypothesis \eqref{Growth} and the fact that
$$
\frac{x}{\lambda(t)} \phi'\Big( \frac{x}{\la(t)} \Big)\in L^{\infty}(\R)\;\;\;\text{uniformly in }\;\;\;t\in[10,\infty),
$$
we have
\[
\abs{\frac{\la'(t)}{\mu(t)\la(t)} \int \frac{x}{\la(t)}  \phi'\Big( \frac{x}{\la(t)} \Big) u(t,x)dx} \lesssim \frac{1}{t^{}\log^{2}t}.
\]
 Similarly,
 \[
\abs{\frac{\mu'(t)}{\mu^2(t)}\int   \phi \Big( \frac{x}{\la(t)} \Big) u(t,x)dx}\lesssim \frac{1}{t \log^2 t}.
 \]

Inserting the above estimates in \eqref{dtI} it follows  that
\be\label{Estimate1}
\frac d{dt}\mathcal V(t) =  \frac1{2\mu(t)\lambda(t)}\int   \phi' \Big( \frac{x}{\la(t)} \Big)  u^2 (t,x)dx +h(t),
\ee
with $\,h\in L^1([10,\infty))$. From this estimate, we clearly see that 
\be\label{Integra0}
\int_{10}^\infty \frac1{\mu(t)\lambda(t)}\int   \phi' \Big( \frac{x}{\la(t)} \Big)  u^2 (t,x)dx dt =c_{\phi}<+\infty.
\ee
for any weight $\phi(\cdot)$ in the class $\mathcal A$ defined in \eqref{class} and since $\frac1{\mu(t)\lambda(t)}$ does not integrate in $[10,\infty)$, it follows that for  some increasing  sequence of time $t_n\to +\infty$,
\be\label{Integra1}
\lim_{n\to \infty} \int_{-\infty}^{\infty} \phi\Big( \frac{x}{\la(t_n)} \Big) u^2(t_n,x)dx =0.
\ee
\medskip

\medskip
\noindent
Next, we define the functional $\mathcal J(t)$ as
\be\label{J}
\mathcal J(t):= \frac1{2\mu(t)}\int \phi\Big( \frac{x}{\la(t)} \Big) u^2(t,x)dx.
\ee
We claim the following result.

\begin{lem}\label{dtJ0}
We have
\be\label{dtJ}
\begin{aligned}
\frac d{dt}\mathcal J(t) = &~ {} -\frac{\mu'(t)}{2\mu^2(t)}\int \phi\Big( \frac{x}{\la(t)} \Big) u^2(t,x)dx   - \frac{\la'(t)}{2\mu(t)\la(t)} \int \frac{x}{\la(t)} \phi'\Big( \frac{x}{\la(t)} \Big) u^2(t,x)dx\\
 &~ {}+ \frac1{\mu(t)} \int  \frac1{\la(t)}\phi\Big( \frac{x}{\la(t)} \Big) u  \left(-\mathcal T_\delta \partial^2_x -\frac{\partial_x}{\delta}\right) u  (t,x)dx\\
 &~ {} + \frac1{3\la(t)\mu(t)} \int \phi'\Big( \frac{x}{\la(t)} \Big) u^3 (t,x)dx.
\end{aligned}
\ee
\end{lem}

\begin{proof} We have, using \eqref{symbol2}, that
\[
\begin{aligned}
\frac d{dt}\mathcal J(t) = &~ {} -\frac{\mu'(t)}{2\mu^2(t)}\int \phi\Big( \frac{x}{\la(t)} \Big) u^2(t,x)dx  - \frac{\la'(t)}{2\mu(t)\la(t)} \int \frac{x}{\la(t)} \phi'\Big( \frac{x}{\la(t)} \Big) u^2(t,x)dx \\
&~{} + \frac1{\mu(t)} \int \phi\Big( \frac{x}{\la(t)} \Big) (u\partial_t u)(t,x)dx\\
= &~ {}  -\frac{\mu'(t)}{2\mu^2(t)}\int \phi\Big( \frac{x}{\la(t)} \Big) u^2(t,x)dx   - \frac{\la'(t)}{2\mu(t)\la(t)} \int \frac{x}{\la(t)} \phi'\Big( \frac{x}{\la(t)} \Big) u^2(t,x)dx\\
 &~ {}+ \frac1{\mu(t)} \int \phi\Big( \frac{x}{\la(t)} \Big) u \Big(-\mathcal T_\delta \partial_x^2  u  -\frac{\partial_xu}{\delta}-  u\partial_xu \Big) (t,x)dx\\
  = &~ {}  -\frac{\mu'(t)}{2\mu^2(t)}\int \phi\Big( \frac{x}{\la(t)} \Big) u^2(t,x)dx   - \frac{\la'(t)}{2\mu(t)\la(t)} \int \frac{x}{\la(t)} \phi'\Big( \frac{x}{\la(t)} \Big) u^2(t,x)dx\\
 &~ {}- \frac1{\mu(t)} \int   \phi\Big( \frac{x}{\la(t)} \Big)  u(t,x)  \,\partial_x\mathcal L_{\delta}(\partial_x) u  (t,x)dx\\
 &~ {} + \frac1{3\la(t)\mu(t)} \int \phi'\Big( \frac{x}{\la(t)} \Big) u^3 (t,x)dx.
\end{aligned}
\]
The proof is complete.
\end{proof}

Next, we procede to bound the terms appearing in Lemma \ref{dtJ0} in the right hand side of \eqref{dtJ}. First, from \eqref{la} we have
\[
\abs{-\frac{\mu'(t)}{2\mu^2(t)}\int \phi\Big( \frac{x}{\la(t)} \Big) u^2(t,x)dx} \lesssim \frac{1}{t^{a+1}\log^2 t},
\] 
which is integrable in $[10,\infty)$. Next, we have again the bound
\[
\abs{\frac{\la'(t)}{2\mu(t)\la(t)} \int \frac{x}{\la(t)} \phi'\Big( \frac{x}{\la(t)} \Big) u^2(t,x)dx}\lesssim \frac{1}{t^{a+1}\log^2 t}.
\]
Now we deal with the term
\be
\label{eqn:care}
A_1(t):= \frac{1}{\mu(t)} \int  \phi\Big( \frac{x}{\la(t)} \Big) u(t,x) \, \mathcal L_{\delta}(\partial_x) \partial_x  u  (t,x)dx,
\ee
which requires more care. Since the symbol $ \mathcal L_{\delta}(\partial_x)$ is even one has that
\be
\label{x1}
\begin{aligned}
A_1(t)&= -  \frac{1}{\mu(t)} \int  \mathcal L_{\delta}(\partial_x) \partial_x\Big(\phi\Big( \frac{x}{\la(t)} \Big) u(t,x) \Big)  u  (t,x)dx\\
&= - \frac{1}{\mu(t)} \int   \phi\Big( \frac{x}{\la(t)} \Big) u(t,x) \, \mathcal L_{\delta}(\partial_x) \partial_x  u(t,x) dx \\
&-\frac{1}{\mu(t)} \int  u(t,x)\,\Big[ \mathcal L_{\delta}(\partial_x) \partial_x; \phi\Big( \frac{x}{\la(t)} \Big)\Big] u(t,x) dx.
\end{aligned}
\ee

Therefore,
\be\label{x2}
2A_1(t) =-\frac{1}{\mu(t)} \int   u(t,x)\,\Big[ \mathcal L_{\delta}(\partial_x) \partial_x; \phi\Big( \frac{x}{\la(t)} \Big)\Big] u(t,x) \, dx.
\ee

Using Fourier transform one sees that
\be
\label{x3}
\begin{aligned}
&\widehat{\Big( \Big [\mathcal L_{\delta}(\partial_x) \partial_x; \phi \Big] f \Big)}(\xi)\\
&=c \int(\Omega_{\delta}(\xi)-\Omega_{\delta}(\eta))\,\widehat{\phi}(\xi-\eta)\,\widehat{f}(\eta) d\eta.
\end{aligned}
\ee
Since
\be
\label{x4}
\Omega_{\delta}(\xi)-\Omega_{\delta}(\eta)=(\xi-\eta)\Omega_{\delta}'(\eta)+\frac{(\xi-\eta)^2}{2} \Omega_{\delta}''(\theta\eta+(1-\theta)\xi)
\ee
with $\theta=\theta(\xi,\eta)\in[0,1]$ and $\|\Omega_{\delta}''\|_{\infty}\leq c_{\delta}$ it follows that
\be\label{x5}
\Big [\mathcal L_{\delta}(\partial_x) \partial_x; \phi \Big] f(x) =c \,\phi'(x)(\Omega_{\delta}'(\partial_x)f)(x)+R_1(x),
\ee
with the term $R_1$ (after combining Plancherel identity, Minkowski  integral inequality and Sobolev embedding) satisfying the estimate
\be
\label{x6}
\| R_1\|_2\leq \| \widehat{\phi''}\|_1\|f\|_2\leq c \|\phi''\|_{1,2}\|f\|_2,
\ee
where $ \,\|h\|_{k,2}=\|(1-\partial_x^2)^{k/2}h\|_2,\;\;k\in\R$.

Inserting this estimate in \eqref{x2}  we can conclude that 
\be
\label{x7}
A_1(t)=  \frac{c}{\mu(t)\,\la(t)}\,\int u(t,x)\,\phi'\Big( \frac{x}{\la(t)}\Big)\Omega_{\delta}'(\partial_x)u(t,x) \, dx + r_1(t)
\ee
with  $r_1\in L^1([10,\infty))$.

From \eqref{x7} we can write, using that the symbol of $q(\partial_x)$ is even \eqref{a7} that
\begin{equation}
\label{001}
\begin{aligned}
\mu(t)\,\la(t)(A_1(t)+r_1(t))&=c \int \phi'\Big( \frac{x}{\la(t)} \Big) u(t,x)\,q(\partial_x) q(\partial_x) u  (t,x) dx\\
&=c\int q(\partial_x)(\phi'\Big( \frac{x}{\la(t)} \Big)  u(t,x))\,q(\partial_x) u  (t,x)dx\\
&=c \int \phi'\Big( \frac{x}{\la(t)} \Big) q(\partial_x) u(t,x)\,q(\partial_x) u  (t,x)dx\\
&\quad + c\int  q(\partial_x)\Big[q(\partial_x); \phi'\Big(\frac{\cdot}{\la(t)}\Big) \Big] u(t,x) \, u  (t,x)dx
\end{aligned}
\end{equation}

We claim :
\begin{equation}
\label{estimate111}
\|\,q(\partial_x)\Big[q(\partial_x); \rho \Big] f\|_2\leq c\,\|\,\partial_x\rho\|_{2,2}\|f\|_2.
\end{equation}

To prove \eqref{estimate111} we take Fourier transform to get
\begin{equation}
\label{aa3}
\widehat{\big(q(\partial_x)\Big[q(\partial_x); \rho \Big] f\big)}(\xi)=c \,q(\xi)\int(q(\xi)-q(\eta))\widehat{\rho}(\xi-\eta)\widehat{f}(\eta)d\eta
\end{equation}
We observe : if $\xi \eta\geq 0$, then
\[
|q(\xi)|\,|q(\xi)-q(\eta)|\leq c |q(\xi)|\,|q'(\theta\eta+(1-\theta)\xi)|\,|\xi-\eta|\leq c|\xi-\eta|\big(1+|\xi-\eta|\big).
\]
see \eqref{a7}-\eqref{a10},  and if $\xi\eta<0$, then
\[
|q(\xi)|\,|q(\xi)-q(\eta)|\leq c|\xi-\eta|\big(1+|\xi-\eta|\big).
\]
Hence,
\begin{equation}
\label{aa5}
\begin{aligned}
&\|\,q(\partial_x)\Big[q(\partial_x); \rho \Big] f\|_2\leq c\|\widehat{((1-\partial_x^2)^{1/2}\partial_x\rho)}\ast \widehat{f}||_2\\
&\leq c\|\widehat{((1-\partial_x^2)^{1/2}\partial_x\rho)}\|_1\|\widehat{f}||_2
\leq c\|(1-\partial_x^2)^{1/2}\partial_x\rho\|_{1,2}\|{f}\|_2\\
&\leq c\|\partial_x\rho\|_{2,2}\|{f}\|_2.
\end{aligned}
\end{equation}
In our case $\rho(x)=\phi'\Big( \frac{x}{\la(t)}\Big)$, therefore
\[
\|q(\partial_x)\Big[q(\partial_x); \phi'\Big(\frac{\cdot}{\la(t)}\Big) \Big] u \|_2\leq \frac{c}{(\la(t))^{1/2}}\|\phi''\|_{1,2}\|u\|_2,
\]
which inserted in \eqref{001} tells us that
\be
\label{aa7}
A_1(t)=\frac1{\mu(t)\la(t)} \int \phi'\Big( \frac{x}{\la(t)} \Big) q(\partial_x) u\,q(\partial_x) u  (t,x)dx+r_1(t)+ g(t),
\ee
with
\[
|g(t)|\leq \frac{c}{\mu(t)(\la(t))^{3/2}}\in L^1([10,\infty)).
\]

Finally, we shall bound the last term in \eqref{dtJ}. Combining  the remark in \eqref{issue}, the calculus for fractional derivatives deduce in \cite{KPV} (Appendix) and the conservation laws in \eqref{CL} we write
\be
\label{777}
\begin{aligned}
\Big|\frac1{\la(t)\mu(t)}& \int \phi'\Big( \frac{x}{\la(t)} \Big) u^3 (t,x)dx\Big|=
\frac{1}{\la(t)\mu(t)} \| \phi_0'\Big( \frac{\cdot}{\la(t)} \Big) u (t,\cdot)\|_3^3\\
&\leq c\,\frac{1}{\la(t)\mu(t)} \| D_x^{1/2}(\phi_0'\Big( \frac{\cdot}{\la(t)} \Big) u (t,\cdot))\|_2 \| \phi_0'\Big( \frac{\cdot}{\la(t)} \Big) u^2 (t,\cdot)\|_2^2\\
&\leq c^*\,\frac{1}{\la(t)\mu(t)}\,\int \phi'_0\Big( \frac{x}{\la(t)} \Big) u^2 (t,x)dx.
\end{aligned}
\ee
with $c^*=c^*(\|u_0\|_{1/2,2};\|\phi'_0\|_{1,2})$. Using \eqref{Integra0} with $\phi_0\in \mathcal A_C$ instead of $\phi\in \mathcal A_C$ it follows that the last term in \eqref{777} belongs to $L^1([10,\infty)$.

 Now gathering the above results one has that
\be\label{Estimate111}
\frac d{dt}\mathcal J(t) =  \frac1{2\mu(t)\lambda(t)}\int   \phi' \Big( \frac{x}{\la(t)} \Big)  (q(\partial_x)u)^2 (t,x)dx +h_1(t),
\ee
with $\,h_1\in L^1([10,\infty))$. Therefore, from the conservation laws in \eqref{CL} one concludes that
\be\label{Integra00}
\int_{10}^\infty \frac1{\mu(t)\lambda(t)}\int   \phi' \Big( \frac{x}{\la(t)} \Big) (q(\partial_x) u)^2 (t,x)dx dt <+\infty.
\ee
This combines with \eqref{Integra0} and \eqref{a8}-\eqref{a10} leads us to the estimate
\be\label{Integra000}
\int_{10}^\infty \frac1{\mu(t)\lambda(t)}\int   \phi' \Big( \frac{x}{\la(t)} \Big)( (q(\partial_x) u)^2 +u^2)(t,x)dx dt \\
<+\infty
\ee
which yields the desired result \eqref{AS} and basically completes the proof of Theorem \ref{TH1}.

\begin{rem}

In \cite{KeMa} the estimate \eqref{777} was established for solutions of the BO equation by fixing the  weight
$$
\phi(x)=\frac{\pi}{2}+\text{arctan}(x)\;\;\;\;\;\;\;\text{so that}\;\;\;\;\;\;\phi'(x)=\frac{1}{1+x^2},
$$
which resembles the profile of the soliton solution of the BO equation. 

\end{rem}

\section{Proof of Theorems \ref{TH1a}-\ref{TH2a}}

\begin{proof}  [Proof of Theorem \ref{TH1a}]

 We observe that multiplying the equation in \eqref{BOg} by $\,x\,$ and integrating  after some integration by part one obtains
 \be
 \label{aa1}
 \frac{d\;}{dt}\,\int_{-\infty}^{\infty}\,x\,u(t,x) dx=\int_{-\infty}^{\infty} f(u(t,x))dx,
\ee
which yields the desired result.

\end{proof}
\begin{proof} [Proof of Theorem \ref{TH2a}]

Multiplying the equation \eqref{BOg} by $\,\mathcal H\partial_xu-(u^k+p_k(u))$ and integration the result one gets 
\be
\label{cl2}
\frac{d\;}{dt}\int_{-\infty}^{\infty}\Big(\frac{(D_x^{1/2}u)^2}{2}-\Big(\frac{u^{k+1}}{k+1}+G_k(u)\Big)\Big)dx=\frac{d\;}{dt} I_3(u)=0,
\ee
where
\be
\label{f1}
G_k(s) =\int_0^s\,p_k(r)dr.
\ee

Now multiplying the equation \eqref{BOg} by $\,x u(t,x)$ we obtain
\be
\label{f2}
\frac{1}{2}\frac{d\;}{dt}\int u^2xdx+\int\mathcal H\partial_x^2u u x dx + \int \,\partial_x(u^k+p_k(u)) u x dx=0.
\ee

First, we consider the second term in \eqref{f2}
\be
\begin{aligned}
\label{f3}
E_1:=&\int\mathcal H\partial_x^2u u x dx= - \int\mathcal H\partial_xu \partial_xu x dx - \int\mathcal H\partial_xu u  dx\\
=&\, E_{1,1}- \int (D_x^{1/2}u)^2dx.
\end{aligned}
\ee

We claim that 
\be
\label{claim1}
\,E_{1,1}= - \int\mathcal H\partial_xu \partial_xu x dx=0.
\ee

By integration by part
\be\label{f4}
\begin{aligned}
E_{1,1}&=\int  \partial_xu \mathcal H (\partial_xu x)dx\\
&=  \int \partial_x u \mathcal H\partial_xu x dx +\int \partial_xu \big[\mathcal H; x\big]\partial_xu dx.
\end{aligned}
\ee

Therefore,
\be
\label{f5}
E_{1,1}=\,-\frac{1}{2}\,\int \, u\,\partial_x \big[\mathcal H; x\big]\partial_xudx
\ee

 To obtain the claim it suffices to see that if 
 \be
 \label{f6}
 \partial_x \big[\mathcal H; x\big]\partial_x u\equiv 0,
 \ee
which follows the identities
\be
\label{f7}
\widehat{\big(\partial_x \big[\mathcal H; x\big]\partial_x u\big)}(\xi,t)=-2\pi(|\xi|\partial_{\xi}\big(\xi \widehat{u}(\xi,t)\big)-\xi\partial_{\xi}\big(|\xi| \widehat{u}(\xi,t)\big))=0.
\ee

Returning to \eqref{f2} we observe that
\be
\label{f8}
\begin{aligned}
\int \,\partial_x(u^k+p_k(u)) u x dx &= \int \partial_x\Big(\frac{k}{k+1} \,u^{k+1}+P_k(u)\Big)(t,x)\, xdx\\
&=-\frac{k}{k+1} \,\int u^{k+1}dx -\int P_k(u)dx,
\end{aligned}
\ee
where $\,\partial_xP_k(s)=s\,\partial_xp_k(s)$.

\vskip.1in

Thus, inserting \eqref{f3}, \eqref{claim1}, and \eqref{f8} in \eqref{f2} one gets that 
\be
\label{f9}
\frac{1}{2} \frac{d\;}{dt}\int u^2(t,x)\,x\,dx-\int\,(D_x^{1/2}u)^2 dx-\frac{k}{k+1} \,\int u^{k+1}dx -\int P_k(u)dx=0
\ee
which gives that
\be
\label{f10}
\frac{1}{2} \frac{d\;}{dt}\int u^2(t,x)\,x\,dx= 2I_3(u_0)+\frac{k+2}{k+1}\int u^{k+1}dx +\int (2G_k(u)+P_k(u))dx.
\ee

Next, we recall the following inequality of Gagliardo-Nirenberg type : for $ \,q\in [2,\infty)$
\be
\label{GN}
\| f\|_q\leq c_q \| D^{1/2-1/q}f\|_2\leq c_q\|f\|_2^{2/q} \,\| D^{1/2}f\|_2^{1-2/q}.
\ee

Using the notation
\be\label{f11}
y(t)=\| D_x^{1/2}u(t)\|_2,
\ee
and combining \eqref{cl2}, \eqref{poly} and \eqref{GN} one sees that
\be
\label{a111}
 \frac{d\;}{dt}\,\int_{-\infty}^{\infty}\,x\,u(t,x) dx=\int_{-\infty}^{\infty} f(u(t,x))dx.
\ee
An integration in the time interval $[0,\omega]$ completes the proof.

\end{proof}

\section{Proof of Theorem \ref{TH2}}

\begin{proof} First, we consider the principal symbol of the operator modeling the dispersive relation in the ILW equation \eqref{ILW}
 \be
\label{3.1}
\sigma( \mathcal T_{\delta}\,\partial_x^2)=  -4 \pi^2 \xi^2\,\text{coth}(2\pi \delta \xi)\,i.
\ee

Thus,
\be
\label{3.2}
\begin{aligned}
&4\pi^2|\xi|\xi -4 \pi^2 \xi^2\,\text{coth}(2\pi \delta \xi)\\
& =4\pi^2|\xi|\xi\Big(1-\text{sgn}(\xi)\text{coth}(2\pi\delta \xi)\Big)\\
&=4\pi^2|\xi|\xi\, \Big(-\frac{2e^{-4\pi\delta |\xi|}}{1-e^{-4\pi\delta |\xi|}}\Big).
\end{aligned}
\ee

Hence, for any $M>0$ sufficiently large (such that $\,e^{-4\pi^2\delta M}\leq 1/2$) one has
\be
\label{3.3}
\sup_{|\xi|\geq M}|\,|4\pi^2|\xi|\xi -4 \pi^2 \xi^2\,\text{coth}(2\pi \delta \xi) |\leq 16 M^2\,e^{-4\pi^2\delta M}.
\ee

Now, taking $\chi \in C_0^{\infty}(\R)$ such that $\text{supp}\subset [-1,1]$ with $\chi(x)=1,\,x\in [-1/2,1/2]$ and $ \chi(x)\geq 0,\;x\in\R.$
one sees from \eqref{3.3} that for any $R>0$
\be
\label{symbl1}
\sigma_{R,1}(\xi)=(4\pi^2|\xi|\xi -4 \pi^2 \xi^2\,\text{coth}(2\pi \delta \xi))\Big(1-\chi(\xi/R)\Big)\in S^{-\infty},
\ee
i.e. $\sigma_{R,1}$ is the symbol of a smoothing pseudo-differential operator and that for any $R>0$
\be
\label{symbl2}
\sigma_{R,2}(\xi)=(4\pi^2|\xi|\xi -4 \pi^2 \xi^2\,\text{coth}(2\pi \delta \xi)) \,\chi(x/R),
\ee
is a multiplier with compact support.

Therefore, the operator $A_{\delta_{R}}(\partial_x)$ with symbol $\sigma(A_{\delta_{R}})(\xi)=4\pi^2|\xi|\xi -4 \pi^2 \xi^2\,\text{coth}(2\pi \delta \xi) $ satisfies that for any $s\geq 0$
\be
\label{estimate1}
\| \,A_{\sigma_R} f\|_{s,2}\leq c\,\|f\|_2,\;\;\;\;\text{with}\;\;\; c=c(s;\delta).
\ee

Thus, rewriting  the ILW equation \eqref{ILW} as
\be\label{moILW} 
\begin{aligned}
 \partial_t u + \mathcal H\,\partial_x^2u+(\mathcal T_\delta \partial_x^2  -\mathcal H\,\partial_x^2) u+ \frac1{\delta}\partial_x u  + u\partial_x u = 0,
\end{aligned}
\ee
we observe that the argument carried out in \cite{ILP2} for the BO equation based on weighted energy estimates can be applied for the equation in \eqref{moILW} without any major modification.

This essentially yields the proof of Theorem \ref{TH2}.
\end{proof}

\section{Proof of Theorem \ref{TH3} and Corollary \ref{TH4}}

We follow the proof of Theorem \ref{TH1} with some key differences.  The most important is that now we choose $\la(t)$ such that $\la^{-1}(t)$ is  \emph{integrable in time}. Set now (compare with \eqref{la}), for $t \geq 10$ and $\epsilon>0$, 
\be\label{la_new}
\la(t) := t \log^{1+\epsilon} t, \quad \mu(t):= \la(t)=  t \log^{1+\epsilon} t.
\ee
Any other choice of $\la(t)$ (and therefore, $\mu(t)$) which is bigger in size, also works. Note that $1/\la(t)$ is now integrable in $[10,\infty)$, and
\[
\frac{\la'(t)}{\la(t)}= \frac{\mu'(t)}{\mu(t)}\sim \frac1t, \quad \frac{\mu'(t)}{\la(t)} \sim \frac{1}{t }.
\]
Recall that now $\frac{\mu'(t)}{\la(t)}$ is not integrable in $[2,\infty)$.

\medskip

\noindent
Now, as in \eqref{J}, we define the modified functional $\mathcal J_e(t)$ as
\be\label{Je}
\mathcal J_e(t):= \frac1{2}\int \phi\Big( \frac{x+\mu(t)}{\la(t)} \Big) u^2(t,x)dx.
\ee
Here, $\phi$ is a smooth bounded function to be chosen later. Following the lines of the proof of Lemma \ref{dtJ0}, we claim
\begin{lem}\label{dtJ22}
We have
\be\label{dtJ2}
\begin{aligned}
\frac d{dt}\mathcal J_e(t) = &~ {} \frac{\mu'(t)}{2\la(t)}\int \phi'\Big( \frac{x+\mu(t)}{\la(t)} \Big) u^2(t,x)dx  \\
&~{}  - \frac{\la'(t)}{2\la(t)} \int \left(\frac{x+\mu(t)}{\la(t)} \right)\phi'\Big( \frac{x+\mu(t)}{\la(t)} \Big) u^2(t,x)dx\\
 &~ {}+ \int \phi\Big( \frac{x+\mu(t)}{\la(t)} \Big) u  \left(-\mathcal T_\delta \partial^2_x -\frac{\partial_x}{\delta}\right) u  (t,x)dx\\
 &~ {} + \frac1{3\la(t)} \int \phi'\Big( \frac{x+\mu(t)}{\la(t)} \Big) u^3 (t,x)dx.
\end{aligned}
\ee
\end{lem}
Now, we prove decay on the left portion $x\sim -\la(t)$. First of all, choose $\phi \in C^\infty \cap L^\infty$ such that $\phi(x)\in [0,1]$ for all $x\in\R$,  $\phi(s)= 1$ if $s\leq -1$, $\phi(s)=0$ for $s\geq 0$, and $\phi'\leq 0$ in $\R$, see \eqref{class}. 

\medskip

Notice that $\phi'\left(\frac{x+\mu(t)}{\la(t)} \right) \neq 0$ only in the region of $x\in\R$ such that  $-1\leq \frac{x+\mu(t)}{\la(t)}\leq 0$, essentially  nonpositive. Therefore, for all $x\in\R,$
\[
 \phi'\Big( \frac{x+\mu(t)}{\la(t)} \Big) \leq 0, \qquad \left(\frac{x+\mu(t)}{\la(t)} \right)\phi'\Big( \frac{x+\mu(t)}{\la(t)} \Big)\geq 0.
\]
Since  
\[\int \phi\Big( \frac{x+\mu(t)}{\la(t)} \Big) u  \left(-\mathcal T_\delta \partial^2_x -\frac{\partial_x}{\delta}\right) u  (t,x)dx = - \int  \phi\Big( \frac{x}{\la(t)} \Big) u(t,x) \, \mathcal L_{\delta}(\partial_x) \partial_x  u  (t,x)dx, 
\]proceeding as in the estimate of \eqref{eqn:care}, we get
\[
\abs{\int  \phi\Big( \frac{x}{\la(t)} \Big) u(t,x) \, \mathcal L_{\delta}(\partial_x) \partial_x  u  (t,x)dx} \lesssim \frac1{\la(t)}.
\]
Now, the cubic term in \eqref{dtJ2} is treated in similar fashion as in \eqref{777}. We obtain that 
 \[
 \abs{\frac1{3\la(t)} \int \phi'\Big( \frac{x+\mu(t)}{\la(t)} \Big) u^3 (t,x)dx} \lesssim \frac1{\la(t)}.
 \]
Consequently, following the proof of Theorem \ref{TH1}, and using that $\la^{-1}(t)$ integrates in time, we have from \eqref{dtJ2}:
\[
\begin{aligned}
&\int_{10}^\infty \frac{1}{t} \left( \int -\phi'\Big( \frac{x+\mu(t)}{\la(t)} \Big) u^2(t,x)dx \right.\\
& \qquad\quad  \left. + \int \left(\frac{x+\mu(t)}{\la(t)} \right)\phi'\Big( \frac{x+\mu(t)}{\la(t)} \Big) u^2(t,x)dx \right) dt<+\infty.
\end{aligned}
\]
Hence, a sequence of times $t_n\uparrow+\infty$ is such that the integrand converges to zero. Once again in \eqref{dtJ2}, choose now $\widetilde\phi \in C_0^\infty(\R) $ such that 
\be\label{tildephi}
\begin{aligned}
& \hbox{$\widetilde\phi(x)\in [0,1]$ for all $x\in\R$,  $\widetilde\phi(s)= 0$  if $s\leq -3/4$, $\widetilde\phi(s)=1$ for $s\geq -1/4$, }\\
& \hbox{and $\widetilde\phi' \geq 0$  in $\R$. }
\end{aligned}
\ee
A new estimate of $\frac d{dt}\mathcal J_e(t) $ gives
\[
\begin{aligned}
\abs{\frac d{dt}\mathcal J_e(t) } \lesssim &~{}  \frac{1}{t}\int \widetilde\phi'\Big( \frac{x+\mu(t)}{\la(t)} \Big) u^2(t,x)dx \\
&~{} + \frac{1}t \int \left(\frac{x+\mu(t)}{\la(t)} \right)\widetilde\phi'\Big( \frac{x+\mu(t)}{\la(t)} \Big) u^2(t,x)dx + \frac{1}{t\log^{1+\epsilon} t}\\
\lesssim &~{} \frac{1}{t}\int \phi'\Big( \frac{x+\mu(t)}{\la(t)} \Big) u^2(t,x)dx + \frac{1}{t\log^{1+\epsilon} t}.
\end{aligned}
\]
The rest of the proof is direct, see e.g. \cite{MuPo1}. The final conclusion follows from the fact that  $-\frac34\leq \frac{x+\mu(t)}{\la(t)} \leq -\frac14$ is equivalent to $x\sim -\mu(t)=-\la(t)$.

\subsection{Proof of Corollary \ref{TH4}} The proof is simple, just a modification of certain aspects of the proof of Theorem \ref{TH3}.

\medskip

Set now, for $t \geq 10$, and $\ve>0$, 
\be\label{la_new_new}
\la(t) := t \log^{1+\epsilon} t, \quad \mu(t):= t.
\ee
Note that $1/\la(t)$ is now integrable in $[10,\infty)$, and
\[
\frac{\la'(t)}{\la(t)}= \frac{\mu'(t)}{\mu(t)}\sim \frac1t, \quad \frac{\mu'(t)}{\la(t)} \sim \frac{1}{t\log^{1+\epsilon} t}.
\]
Recall that now $\frac{\mu'(t)}{\la(t)}$ \emph{is integrable} in $[2,\infty)$.

\medskip

\noindent
Now we consider the same modified functional $\mathcal J_e(t)$ as in \eqref{Je}. Choose $\phi \in C^\infty \cap L^\infty$ such that $\phi(x)\in [0,1]$ for all $x\in\R$,  $\phi(s)= 0$ if $s\leq 0$, $\phi(s)=1$ for $s\geq 1$, and $\phi'\geq 0$ in $\R$. 

\medskip

Notice that $\phi'\left(\frac{x+\mu(t)}{\la(t)} \right) \neq 0$ only in the region of $x\in\R$ such that  $0\leq \frac{x+\mu(t)}{\la(t)}\leq 1$, essentially  nonnegative. This region reads
\[
- t  \leq x \leq  t\log^{1+\epsilon} t - t,
\]
hence it contains the region where solitons exist. Therefore, for all $x\in\R,$
\[
 \left(\frac{x+\mu(t)}{\la(t)} \right)\phi'\Big( \frac{x+\mu(t)}{\la(t)} \Big)\geq 0.
\]
Proceeding as in the proof of Theorem \ref{TH3}, we obtain now the weaker condition
\be\label{PP}
\int_{10}^\infty \frac{1}{t}  \left( \int_{-t}^{t\log^{1+\epsilon} t-t}\left( \frac{x+\mu(t)}{\la(t)} \right)\phi'\Big( \frac{x+\mu(t)}{\la(t)} \Big) u^2(t,x)dx \right) dt<+\infty.
\ee
This proves \eqref{LL}.

\begin{rem}[Final remark]\label{Final}
Estimate \eqref{LL} also reveals that the choice of  $\tilde\phi$ in \eqref{tildephi} is in some sense \emph{necessary} for having truly decay. Indeed, from \eqref{PP}, having 
\[
\frac{x+\mu(t)}{\la(t)} \gtrsim \frac{1}{\log t}  \implies x \gtrsim - t + t \log^{\epsilon} t \sim t \log^{\epsilon} t ;
\]
and now \eqref{PP} becomes the more tractable integral estimate
\[
\int_{10}^\infty \frac{1}{t\log t}  \left( \int_{t \log^{\epsilon} t }^{t\log^{1+\epsilon} t-t}\phi'\Big( \frac{x+\mu(t)}{\la(t)} \Big) u^2(t,x)dx \right) dt<+\infty.
\]
Now this integral does not contain the bad term $\frac{x+\mu(t)}{\la(t)}$, but we have lost all the solitonic region in the estimate of integrability in time, a property that makes sense with Theorem \ref{TH3}. This remark is also valid for KdV, mKdV, BBM and BO.
\end{rem}


\bigskip

\begin{thebibliography}{9}

\bibitem{ABFS} L. Abdelouhab, J.-L. Bona, M. Felland, and J.-C. Saut, \emph{Nonlocal Models for nonlinear, dispersive waves}, Physica D \textbf{40} (1989) 360--392.

\bibitem{A} J. P. Albert, \emph{Positivity properties and uniqueness of solitary wave solutions of the intermediate long wave equation}, in Evolution equations (Baton Rouge, LA, (1992)), 11-20, Lecture Notes in Pure and Appl. Math., 168, Dekker, New York, 1995.

\bibitem{AT} J. P. Albert and J.F. Toland, {\it On the exact solutions of the intermediate long-wave equation }, Diff. Int. Eq. {\bf 7} (3-4) (1994), 601-612.

\bibitem{Be}  T. B. Benjamin, \emph{Internal waves of permanent
form in fluids of great depth},
J. Fluid Mech. {\bf 29} (1967) 559--592.
\bibitem{BLS} J. L. Bona,  D. Lannes and J.-C. Saut,
{\it Asymptotic models for internal waves}, J. Math. Pures. Appl. {\bf 89} (2008) 538-566.



\bibitem{CL} H. H. Chen, and Y. C. Lee, \emph{Internal-Wave Solitons of Fluids with Finite Depth}, Phys. Rev. Lett. {\bf 43} (1979) 264

\bibitem{CGK} W. Craig, P. Guyenne, and H. Kalisch,
{\it Hamiltonian long-wave expansions for free surfaces and
interfaces}, Comm. Pure. Appl. Math. {\bf 58} (2005)1587-1641.




\bibitem{ILP1}  P. Isaza, F. Linares, and G.Ponce, \emph{On the propagation of regularity and decay of solutions to the $k$-generalized Korteweg-de Vries equation},  Comm. PDE, \textbf{40} (2015) 1336--1364.




\bibitem{ILP2}  P. Isaza, F. Linares, and G. Ponce, \emph{On the propagation of regularities in solutions of the Benjamin-Ono equation},   J. Funct. Anal.  \textbf{270} (2016) 976--1000.





\bibitem{Jo} R. I. Joseph, \emph{Solitary waves in a finite depth fluid}, J. Phys. A {\bf 11} (1978) L97.

\bibitem{JE} R. I. Joseph,and R. Egri \emph{Multi-soliton solutions  in a finite depth fluid}, J. Phys. A 10 (1977) L225

\bibitem{KeMa} C. E. Kenig, and Y. Martel, \emph{Asymptotic stability of solitons for the Benjamin-Ono equation}, Rev. Mat. Iberoam. {\bf 25} (2009) 909--970. 

\bibitem{KPV} C. E. Kenig, G. Ponce, and L. Vega, \emph{Well-posedness and scattering results for the generalized Korteweg-de Vries equation via contraction principle}, Comm. Pure Appl. Math. {\bf 46} (1993), 527--620.

\bibitem{KPW} J. Klipfel, P. A. Perry and Yulin Wu, \emph{Work in progress}.

\bibitem{KAS} Y. Kodama, M. J. Ablowitz, and J. Satsuma, \emph{Direct and inverse inverse scattering problems of the nonlinear intermediate long wave equations}
J. Math. Phys. {\bf 23} (1982) 564--576. 

\bibitem {KSA} Y. Kodama, J. Satsuma , and M. J. Ablowitz, \emph{Nonlinear intermediate long-wave equation: analysis and method of solution}, Phys. Rev. lett.  (40) (1981), 687-690.

\bibitem{KdV} D. J. Korteweg and G. de Vries
  \emph{On the change of form of long waves advancing in a
   rectangular canal, and on a new type of long stationary waves}, 
  Philos. Mag. 
 { \bf 39}
  (1895), 
  422--443.
\bibitem{KMM} M. Kowalczyk, Y. Martel, and C. Mu\~noz, \emph{Kink dynamics in the $\phi^4$ model: asymptotic stability for odd perturbations in the energy space}, J. Amer. Math. Soc. {\bf 30} (2017), 769--798.


\bibitem{KMM1} M. Kowalczyk, Y. Martel, and C. Mu\~noz, \emph{Nonexistence of small, odd breathers for a class of nonlinear wave equations}, Letters in Mathematical Physics, (2017) Vol. 107, Issue 5, 921--931.

\bibitem{KM} C. Kwak, and C. Mu\~noz, \emph{Extended decay properties for generalized BBM equations}, to appear in ``Nonlinear Dispersive Partial Equations and Inverse Scattering", Peter D. Miller, Peter Perry, Jean-Claude Saut, Catherine Sulem Eds, Fields Institute Communications {\bf 83}, Springer (2019).

\bibitem{KKD} T. Kubota, D. R. S. Ko, and L. D. Dobbs, \emph{Weakly nonlinear, long internal gravity waves in stratified fluids of finite depth}, J. Hydronautics {\bf 12} (1978) 157--165.


\bibitem{KMPP} C. Kwak, C. Mu\~noz, F. Poblete, and J. C. Pozo, \emph{The scattering problem for the Hamiltonian abcd Boussinesq system in the energy space}, J. Math. Pures Appl. (9) 127 (2019), 121--159.

\bibitem{LR} D. R Lebedev and A. O. Radul, {\it Generalized internal long
waves equations, construction, Hamiltonian structure, and
conservation laws}, Commun. Math. Phys. {\bf 91} (1983)
543-555.


\bibitem{Mat3} Y. Matsuno, {\it Bilinear Transformation Method}, Academic
Press, New York, 1984.


\bibitem{MoVe} L. Molinet, and S. Vento, \emph{Improvement of the energy method for strongly nonresonant dispersive equations and applications}, Anal. PDE {\bf 8} (2015), no. 6, 1455--1495

\bibitem{MuPo1} C. Mu\~noz and G. Ponce, \emph{Breathers and dynamics of solutions in KdV type equations}, Comm. Math. Phys. {\bf 367} (2019), 581--598.

\bibitem{MuPo2} C. Mu\~noz and G. Ponce, \emph{On the asymptotic behavior of solutions to the Benjamin-Ono equation}, to appear in Proc. Amer. Math. Soc. arXiv:1810.02329.



\bibitem{On} H. Ono, \emph{Algebraic solitary waves on stratified fluids},  J. Phy. Soc. Japan
 {\bf 39} (1975) 1082--1091.

\bibitem{SAF} P. Santini, M. J. Ablowitz, and A.S. . Fokas, \emph{On the limit from the intermediate long wave equation to the Benjamin-Ono equation}, J. Math. Phys. {\bf 25} (4) (1984), 892-899.

\bibitem{SAK} J. Satsuma,  M. J. Ablowitz, and Y. Kodama, \emph{On an internal wave equation describing a stratified fluid with finite depth}
Phys. Lett. A, {\bf 73} (1979) 283. 


\bibitem{Sa} J.-C. Saut, \emph{Benjamin-Ono and intermediate long 
wave equations : modeling, IST and PDE}, arXiv : 1811.08652v1 21 Nov 2018,   to appear in ``Nonlinear Dispersive Partial Equations and Inverse Scattering", Peter D. Miller, Peter Perry, Jean-Claude Saut, Catherine Sulem Eds, Fields Institute Communications {\bf 83}, Springer (2019).

\bibitem{X} Li Xu, {\it Intermediate long waves systems for internal waves},  Nonlinearity {\bf 25} (2012), 597--640.
%
%
%
%
%
%
%
%
%
%
%
%
%
%
%
%
%
%
%
%
%
%
%
%
%
%
%
%
%
%
%
%
%
%
%
%
%
%
%
%
%
%
%
%
%
%
%
%
%
%
%
%
\end{thebibliography}
\end{document}